\documentclass[12pt,a4paper, reqno]{amsart}
\usepackage[abbrev]{amsrefs}
\usepackage{amssymb}
\usepackage{enumerate}
\usepackage{mathtools}

\theoremstyle{plain}
 \newtheorem{theorem}{Theorem}[section]
 \newtheorem{lemma}[theorem]{Lemma}
 
 \newtheorem{proposition}[theorem]{Proposition}
 
\theoremstyle{definition}
 \newtheorem{definition}[theorem]{Definition}
 \newtheorem{example}[theorem]{Example}
\theoremstyle{remark}
 \newtheorem{remark}[theorem]{Remark}
 \newtheorem*{ack}{Acknowledgment}

\numberwithin{equation}{section}

\def\Hess{\mathop{\mathrm{Hess}_{\mathbb{R}^{n-1}}}}
\def\div{{\mathop{\mathrm{div}}\nolimits}_M}
\begin{document}
\title[]{Riemannian starshape and capacitary problems}
\author{Kazuhiro Ishige, Paolo Salani and Asuka Takatsu}
\date{}
\vspace{-15pt}
\begin{abstract}
We prove the Riemannian version of a classical Euclidean result:
every level set of the capacitary potential of a starshaped ring is starshaped.
In the Riemannian setting, we restrict ourselves to starshaped rings in a warped product of an open interval and the unit sphere.
We also extend the result by replacing the Laplacian with the $q$-Laplacian.
\end{abstract}
\maketitle

\vspace{40pt}

\noindent Addresses:

\smallskip
\noindent 
K. I.: Graduate School of Mathematical Sciences, The University of Tokyo, 3-8-1 
Komaba, Meguro-ku, Tokyo 153-8914, Japan\\
\noindent 
E-mail: {\tt ishige@ms.u-tokyo.ac.jp}\\

\smallskip
\noindent 
P. S.: Dipartimento di Matematica e Informatica ``U. Dini", 
Universit\`a di Firenze, viale Morgagni 67/A, 50134 Firenze\\
\noindent 
E-mail: {\tt paolo.salani@unifi.it}\\

\smallskip
\noindent 
A. T.: Department of Mathematical Sciences, Tokyo Metropolitan University, 1-1
Minami-osawa, Hachioji-shi, Tokyo 192-0397, Japan\\
\noindent 
E-mail: {\tt asuka@tmu.ac.jp}\\
\newpage
\section{Introduction}
It is a classical subject in PDEs to determine how the shape of the domain influences the shape of solutions. 
Many papers deal with the question whether and how some relevant geometric properties of the domain (and of the boundary data) of Dirichlet elliptic  problems are inherited by the solution. 
For instance, a prototypal example of results in this direction is the following.
\begin{proposition}\label{theorem:1.1}
Let $\Omega_0$ and $\Omega_1$ be two bounded open sets in $\mathbb{R}^{n}$ with $n\geq 2$ such that $0\in\overline{\Omega_1}\subset\Omega_0$, 
and consider the capacitary potential $u$ of the ring  shaped condenser $\Omega_0\setminus\overline{\Omega_1}$, 
i.e.,  $u$ is the solution to
\begin{equation}\label{eq:1.1}
\Delta_{\mathbb{R}^{n}} u=0\quad\text{ in }\Omega_0\setminus\overline{\Omega_1},\qquad
u=0\quad \text{ on }\partial\Omega_0,\qquad
u=1\quad\text{ in }\overline{\Omega_1}.
\end{equation}
If $\Omega_0$ and $\overline\Omega_1$ are both starshaped about $0$, then all the superlevel sets of $u$ are starshaped about $0$ as well.
\end{proposition}
See~\cite{St}*{Theorem~1} for $n=3$, and then \cites{Kaw1, Kaw2, Kaw3, Marcus} and references therein. 
An analogous result holds for the Green function of a starshaped set (see \cite{Gergen}).

For the reader's convenience, let us recall the Euclidean notion of starshapedness, a simple, yet interesting and important geometric property.
A set $S\subset\mathbb{R}^{n}$ containing  the origin $0$ is said \emph{starshaped about $0$} 
(simply \emph{starshaped} from now on, when there is no possibility of confusion)
 if the whole segment joining any point in $S$ to $0$  is contained in~$S$, 
i.e.,  if $v\in S$ implies $tv\in S$ for every $t\in[0,1]$.
 Clearly, one can easily define starshapedness with respect to any point simply by translation (and Proposition \ref{theorem:1.1} holds the same when substituting $0$ with any point in $\Omega_1$), but we do not need this sophistication here.

Proposition \ref{theorem:1.1} has been extended in several ways, by generalizing the operator involved 
(see for instance \cites{A, BM, F, FG, GR, JKS, Kaw1, Kaw2, Kaw3, Pf, S})
 and also considering analogous problems in Carnot groups \cites{DGS,DG,FP}.
In this paper we establish a similar result to Proposition \ref{theorem:1.1} in a Riemannian manifold, once a natural notion
of starshapedness has been introduced. Furthermore, we generalize our result to the case of $q$-Laplacian.

To this aim, we give the definition of \emph{starshaped neighborhood}.
Roughly speaking, starshaped neighborhoods of a point  are the images through the exponential map of starshaped sets in the tangent space  at the point
(see Definition~\ref{theorem:2.1} for details). 
We prove an interesting characterization of starshaped neighborhood, similar to the Euclidean case
(see Proposition \ref{theorem:2.6}).

The main results of this paper are stated in Section~\ref{section:5},
where we consider the generalization of problem \eqref{eq:1.1} to the $q$-Laplacian (and then some further generalizations) in 
a warped product of an open interval and the unit sphere (for a similar profitable use of this warped product,  see \cite{IST}).
We find sufficient conditions such that, 
if $\Omega_0$ and $\Omega_1$ are starshaped neighborhoods of the same point $o\in\Omega_1$ and $\overline{\Omega_1}\subset\Omega_0$
(in such a case, we say that $\Omega_0\setminus\overline{\Omega_1}$ is a \emph{starshaped ring} about $o$), 
then all the superlevel sets of the solution to the analogous problem to~\eqref{eq:1.1} are  starshaped neighborhoods of $o$ as well.
The idea of the proof is to define a \emph{quasi-starshaped envelope} of the solution and to prove that in fact it coincides with the solution itself via the viscosity comparison principle.
The rest of this paper is organized as follows.  
In Section~\ref{section:2}, we review starshaped neighborhoods and some of their relevant properties in Riemannian manifold. 
In Section~\ref{section:3}, we give some preliminary facts about a warped product of an open interval and the unit sphere.
In Section~\ref{section:4}, we recall briefly the notion of viscosity solutions. 
Finally, in Section~\ref{section:5}, we state and prove our main results.
\section{starshaped neighborhoods}\label{section:2}
Throughout this paper, let 
$n\in \mathbb{N}$ with $n\geq 2$ and
$(M, g)$ be an $n$-dimensional smooth, complete, connected Riemannian manifold.
We denote by  $\nabla_M$, $\div$, ${\mathop{\mathrm{Hess}}\nolimits}_M$ and $\Delta_M$ 
the gradient, the divergence, the Hessian and the Laplacian on $M$, respectively.
For $q\in \mathbb{R}$ with $q\geq 2$, we define the $q$-Laplacian of a function $f$ on $M$ by 
\[
\Delta_{q,M}f\coloneqq{\mathop{\mathrm{div}}\nolimits}_M ( |\nabla_M f|_g^{q-2} \nabla_M f  )
\quad \text{on }\{p\in M \mid  \nabla_M f(p) \neq 0\}.
\]
We then have
\[
\Delta_{q,M}f=|\nabla_M f|_g^{q-4} \left[(q-2) {\mathop{\mathrm{Hess}}\nolimits}_M f(\nabla_M f, \nabla_M f)+ |\nabla_M f|_g^{2} \Delta_M f \right] .
\]

For $o\in M$ and $R>0$, we define
\[ 
B_o(R)\coloneqq \{p\in M \mid d_M(o,p)<R\},
\]
where  $d_M$ is the Riemannian distance function on $M$. 
We also set $B_o(\infty)\coloneqq M$.

For a tangent vector $v$ to~$M$, set $|v|_g\coloneqq g(v,v)^{1/2}$.
We denote by $\langle\cdot, \cdot \rangle$ and~$|\cdot|$ the Euclidean inner product  
and the Euclidean norm, respectively.
With the customary abuse of notation, the same symbol $0$ is used for the origin in any vector space.
\begin{definition}\label{theorem:2.1}
Let $\Omega$ be an open neighborhood  of $o$ in $M$.
We say that $\Omega$ is a {\it normal} neighborhood
if there exists an open neighborhood $S$ of $0$ in $T_oM$ such that  
$S$ is diffeomorphic to $\Omega$ under the exponential map $\exp_o$ at $o$.
Moreover, if $S$ is starshaped about~$0$,
then $\Omega$ is called a {\it starshaped neighborhood of $o$}.
\end{definition}
Note that a starshaped neighborhood is often called a normal neighborhood (see for instance O'Neill~\cite{Oni}*{Chapter~3:~The Exponential Map}).
However, here we use the expression starshaped neighborhood to emphasize the analogy with the Euclidean setting.
\begin{definition}
Let $\Omega$ be a normal neighborhood  of $o$ in $M$ and $S$ an open neighborhood of $0$ in $T_oM$ such that  
$S$ is diffeomorphic to $\Omega$ under $\exp_o$.
\begin{itemize}
\setlength{\leftskip}{-15pt}
\item
We denote by  $\log_o\colon   \Omega \to S$ the inverse map of the restriction of $\exp_o$ to $S$. 
For $p\in\Omega$,  define a curve $\gamma_p:\mathbb{R} \to M$ by
\[
\gamma_p(t)\coloneqq \exp_o( t\log_o (p)).
\]
\item
We say that $\Omega$ is \emph{regular} 
if $\exp_o$ is injective  on $\overline{S}$ and  the interior of $\overline{S}$ coincides with~$S$. 
\end{itemize}
\end{definition}%
Let $\Omega$ be a starshaped  neighborhood of $o$, then  $\gamma_p([0,1])\subset\Omega$ clearly holds for every $p\in\Omega$.
In addition, for each $p\in \Omega\setminus\{o\}$,
there exists a unique $T_{\Omega,p}\in (1,\infty]$ such that   $\gamma_p(t)\in\Omega$ for $t\in [0,T_{\Omega,p})$
together with either $T_{\Omega,p}=\infty$ or $\gamma_p(T_{\Omega,p})\in \partial \Omega$.
Precisely, it is
\[
T_{\Omega,p}=\sup\{t\geq 1 \mid \gamma_p(t)\in \Omega\}.
\]
For $p\in\partial\Omega$, we can coherently set  $T_{\Omega,p}=1$, 
and then, if $\Omega$ is bounded, we have
\[
\Omega=\bigcup_{p\in \partial \Omega}\gamma_p([0,1)), \quad
\overline\Omega=\bigcup_{p\in \partial \Omega}\gamma_p([0,1]).
\]
Moreover, for $p,p'\in \Omega\setminus\{o\}$ and $t>0$ with $t\leq T_{\Omega,p}$ and $t^{-1}\leq T_{\Omega,p'}$,
$p'=\gamma_p(t)$ holds  if and only if $p=\gamma_{p'}(t^{-1})$ holds.

Let us make some examples to illustrate properties of  starshaped neighborhoods.
\begin{example}
Let us consider $M\coloneqq (\mathbb{R}/\mathbb{Z})\times \mathbb{R}$ and fix $o\in M$.  Then $T_o M=\mathbb{R}\times \mathbb{R}$.
\begin{itemize}
\setlength{\leftskip}{-15pt}
\item
We first  give an example of a non-regular starshaped neighborhood.
Define 
\[
S\coloneqq \left\{  (v_1,v_2) \in T_0M \biggm|  v_1^2+v_2^2<\left(\frac12\right)^2\right\}.
\]
Then $\Omega=\exp_o(S)$ is a starshaped neighborhood of $o$
but this is not regular
since $\exp_o$ is not  injective on $\partial S$.
Here $\overline{\Omega}$ is not a manifold with boundary.
\item
Although the closure of  a regular starshaped neighborhood with nonempty smooth boundary  
is  a manifold with boundary, 
there exists a starshaped neighborhood $\Omega$ such that $\overline{\Omega}$ is a manifold with boundary 
but $\Omega$ is not regular.
Indeed, define 
\begin{align*}
S_+&\coloneqq  \left\{ (v_1,v_2)\biggm|  v_1\in\left(\frac14,\frac12\right),\ v_2< \frac{-1}{(v_1-\frac14)(v_1-\frac34)}\right\},\\
S_-&\coloneqq  \left\{ (v_1,v_2)\biggm|  v_1\in\left(-\frac12,-\frac14\right),\ v_2< \frac{-1}{(v_1+\frac14)(v_{1}+\frac34)}\right\},\\
S&\coloneqq \left(\left[-\frac14,\frac14\right]\times \mathbb{R}\right) \cup S_+\cup S_-.
\end{align*}
Then $\Omega=\exp_o(S)$ is a starshaped neighborhood of $o$ and  $\overline{\Omega}$ is  a manifold with boundary.
However  $\exp_o$ is not  injective on $\partial S$ hence $\Omega$ is  not regular.
\item
In Euclidean space, the  union of any two starshaped sets is again a starshaped set.
However, the union of two starshaped neighborhoods is not necessarily a starshaped neighborhood.
Indeed, define 
\begin{align*}
S_0\coloneqq \left\{ (v_1,v_2) \biggm|  v_1^2+v_2^2<\frac{4}{25} \right\}, \quad
S_1\coloneqq \left\{  (v_1,v_2) \biggm|  \left(v_1-\frac13\right)^2+v_2^2<\frac{4}{25}\right\}.
\end{align*}
Then $\Omega_i=\exp_o(S_i)$
is a  regular starshaped neighborhood of $o$ for $i=0,1$ but  $\Omega_0 \cup \Omega_1$ is not a starshaped neighborhood of $o$.
\end{itemize}
\end{example}
Next, for a regular normal neighborhood $\Omega$ of $o$ in $M$,
we consider the relation between  the outward normal unit vector  to $\partial \Omega$ and that to~$\partial \log_o(\Omega)$.
\begin{definition}\label{theorem:2.4}
Let $\Omega$ be a  regular normal neighborhood  of $o$ in $M$ and $S\coloneqq \log_o(\Omega)$.
Assume that the boundary $\partial \Omega$ is nonempty  and smooth.
\begin{itemize}
\setlength{\leftskip}{-15pt}
\item
Let  $p\in \partial \Omega$ and $\nu \in T_pM$.
We say that $\nu$ is \emph{outward} to $\Omega$ at $p$ if
\[
\exp_p(-\varepsilon \nu) \in \Omega\text{ and }
\exp_p(\varepsilon \nu) \notin \Omega\quad \text{ for $\varepsilon>0$ small enough}.
\]
We say that $\nu$ is \emph{normal}  to $\partial \Omega$ at $p$ if 
\[
g(\nu, w)=0\quad\text{ for any $w\in T_p (\partial \Omega)$}.
\]
\item
Let  $v\in \partial S$ and $\nu_o \in T_oM$.
We say that $\nu_o$ is \emph{outward} to $S$ at $v$ if
\[
v- \varepsilon \nu_o \in S
\text{ and }
v+ \varepsilon \nu_o \notin S\quad\text{ for $\varepsilon>0$ small enough}.
\]
We say that $\nu_o$ is \emph{normal} to $\partial S$ at $v$ if 
\[
g(\nu_o, w)=0\quad \text{ for any $w$ tangent to $\partial S$ at $v$ in $T_oM$}.
\]
\end{itemize}
For $p \in \partial \Omega$, 
there exist a unique outward normal unit vector  to $\partial \Omega$ at $p$ and a unique outward normal vector to
$\partial S$ at $\log_o (p)$.
We denote such vectors by $\nu(p)$ and $\nu_o(\log_o (p))$, respectively.
\end{definition}
\begin{lemma}\label{theorem:2.5}
Let  $\Omega$ be a regular normal neighborhood of $o$ in $M$ such that $\partial \Omega$ is nonempty and smooth.
For $p\in \partial \Omega$,
$\nu(p)$ is parallel to
\begin{equation}\label{eq:2.1}
\frac{d}{d t}\bigg|_{t=0}  \exp_o \left( \log_o(p)+t \nu_o(\log_o(p))  \right)
\in T_p M
\end{equation}
with the same direction.
\end{lemma}
\begin{proof}
Take  $p\in \partial \Omega$.
Let  $\xi=(\xi^1, \ldots \xi^{n})$ be the  normal coordinate functions on $\Omega$ at $o$ in~$M$
determined by an orthonormal basis $e_1, \ldots, e_{n}$  of $T_oM$, that is,
\[
\log_o(p')=\sum_{i=1}^{n} \xi^i (p')e_i \quad\text{ for }p'\in \Omega.
\]
Since $\overline{\Omega}$ is a manifold with boundary,
there exists a diffeomorphism $\eta=(\eta^1, \ldots \eta^{n})$ from an open neighborhood $O$ of $p$ in $\overline{\Omega}$ to an open set in $\mathbb{R}^{n-1}\times [0,\infty)$ such that 
\[
\eta^{n}(\Omega \cap O)>0, \qquad \eta^{n}(\partial \Omega\cap O)=0, \qquad \eta(p)=0.
\]
Then  $\nu(p)$ is parallel to
\[
-\frac{\partial}{\partial \eta^{n}}\bigg|_p 
= -\sum_{i=1}^{n}\frac{\partial \xi^i}{\partial \eta^{n}}(p)\frac{\partial }{\partial \xi^i}\bigg|_p
\in T_pM
\]
with the same direction.
On the other hand, since  $y\in \eta( \partial \Omega \cap O)$ satisfies 
\[
\log_o( \eta^{-1}(y) )=\sum_{i=1}^{n} \xi^{i}(\eta^{-1}(y)) e_i \in \partial S \cap \log_o(O), 
\]
$\nu_o(\log_o(p))$ is parallel to 
\[
-\sum_{i=1}^{n} \frac{\partial \xi^i}{\partial \eta^{n}}(p)  e_i  \in T_oM
\]
with the same direction.
Define curves in $T_oM$ by 
\[
c^i(t)\coloneqq  \log_o(p)+t e_i\quad\text{ for }i=1,\ldots, n,
\qquad
c(t)\coloneqq  \log_o(p)-t \sum_{i=1}^{n} \frac{\partial \xi^i}{\partial \eta^{n}}(p) e_i ,
\quad
\text{for }t\in \mathbb{R}.
\]
Then the tangent vector defined in \eqref{eq:2.1}  is parallel to $d\exp_o(\dot{c}(0))$
with the same direction. 
For a smooth function $f$ defined around~$p$, we see that 
\begin{align*}
  -\frac{\partial}{\partial \eta^{n}}\bigg|_p  f
&= -\sum_{i=1}^{n}\frac{\partial \xi^i}{\partial \eta^{n}}(p)
\frac{d}{dt}\bigg|_{t=0}  \left( f \circ \exp_o \right)( c^i(t))\\
&
= -\sum_{i=1}^{n}\frac{\partial \xi^i}{\partial \eta^{n}}(p)
df \left(d\exp_o \left(  \dot{c}^i(0)  \right)\right)\\
&
=
df\left(d\exp_o\left(-\sum_{i=1}^{n}\frac{\partial \xi^i}{\partial \eta^{n}}(p)   \dot{c}^i(0)  \right)\right)\\
&=df \left(d\exp_o\left( \dot{c}(0)  \right)\right)
=\left(d\exp_o\left( \dot{c}(0)  \right)\right)f.
\end{align*}
This completes the proof of the lemma.
\end{proof}
Now we give a characterization of  a regular starshaped neighborhood  in terms of the normal vector.
\begin{proposition}\label{theorem:2.6}
Let $\Omega$ be a regular normal neighborhood  of $o$ in $M$ such that $\partial \Omega$ is nonempty and smooth.
Then $\Omega$ is  a  starshaped neighborhood of $o$ if and only if 
\[
g(\nu(p), \dot{\gamma}_p(1))\geq 0
\]
holds for $p\in \partial \Omega$.
\end{proposition}
\begin{proof}
Let $p\in \partial \Omega$.
Setting
\[
c(t,\varepsilon)\coloneqq t (\log_{o}(p)+\varepsilon \nu_o(\log_o(p)))\in T_o M
\quad\text{ for }t,\varepsilon \in \mathbb{R},
\]
we observe from the Gauss lemma~(see for instance~\cite{Oni}*{Lemma~5.1}) that 
\[
g\left(d\exp_o\left( \frac{\partial}{\partial \varepsilon}\bigg|_{(1,0)} c\right), 
d\exp_o\left( \frac{\partial}{\partial t}\bigg|_{(1,0)} c\right) \right)
=g(\nu_o(\log_o(p)),\log_o(p)). 
\]
It turns out that 
\begin{align*}
d\exp_o\left( \frac{\partial}{\partial \varepsilon}\bigg|_{(1,0)} c \right)
&=\frac{\partial}{\partial \varepsilon}\bigg|_{(1,0)} \exp_o( c(t,\varepsilon)) 
=
\frac{d}{d \varepsilon}\bigg|_{\varepsilon=0}  \exp_o \left( \log_o(p)+\varepsilon \nu_o(\log_o(p))  \right),\\
d\exp_o\left( \frac{\partial}{\partial t}\bigg|_{(1,0)} c\right)
&=\frac{\partial}{\partial t}\bigg|_{(1,0)} \exp_o( c(t,\varepsilon) )=\dot{\gamma}_p(1).
\end{align*}
These with Lemma~\ref{theorem:2.5} imply that 
$g(\nu(p), \dot{\gamma}_p(1))$ and $g(\nu_o(\log_o(p)),\log_o(p))$ 
have the same  sign.

Assume that $\Omega$ is  a  starshaped neighborhood of $o$.
It follows from \cite{DF}*{Theorem~2.1~(i)} that 
\begin{equation}\label{eq:2.2}
g(\nu_o( \log_o(p)), \log_o(p)) \geq 0
\quad \text{for }p\in \partial \Omega.
\end{equation}
which yields $g(\nu(p), \dot{\gamma}_p(1)) \geq 0$.
Thus the proof of  the `if' part of the assertion follows.

Let us prove the `only if' part of the assertion.
Assume that $g(\nu(p), \dot{\gamma}_p(1))\geq 0$ holds for $p\in \partial \Omega$,
which in turn (again thanks to Lemma~\ref{theorem:2.5}) shows \eqref{eq:2.2}.
Set $S\coloneqq \log_o(\Omega)$.
For  a unit tangent vector $v\in T_oM$, set 
\[
\lambda_v\coloneqq \sup \{ \Lambda  >0 \mid \lambda v \in  S \ \text{ for }\lambda \in (0,\Lambda)\}.
\]
In order to show that $\Omega$ is  a  starshaped neighborhood of $o$
(that is, $S$ is starshaped about~$0$),
we have only to prove 
that, for any unit tangent vector $v\in T_oM$, either
$\lambda_v=\infty$ or $\lambda v \notin S$ for $\lambda>\lambda_v$.
Assume that  a unit tangent vector $v\in T_oM$ satisfies $\lambda_v<\infty$
and 
\[
\Lambda_v\coloneqq \inf\{ \lambda>\lambda_v \mid \lambda v \in S\}>\lambda_v.
\]
We find that $\Lambda_v v \in \partial S$ and  
there exists $\varepsilon>0$ such that 
\begin{equation}\label{eq:2.3}
\Lambda_v v- \delta\nu_o( \Lambda_v v ) \in S
\quad
\text{for }\delta\in (0,\varepsilon).
\end{equation}
If follows from \eqref{eq:2.2} for $p=\exp_o(\Lambda_v v)$ that $g(\nu_o( \Lambda_v v), \Lambda_v v) \geq 0$,
which implies that $v$ is not parallel to $\nu_o( \Lambda_v v)$.
We fix an orthonormal basis  $e_1, \ldots, e_{n}$ of $T_oM$ and identify $T_oM$ with $\mathbb{R}^{n}$.
Without loss of generality,  we can choose $e_{n}=-\nu_o(\Lambda_v v)$.
For $x \in T_oM$, set $x^\perp\coloneqq x-g(x, e_{n})e_n$.
By the smoothness of $\partial S$, 
there exists an open neighborhood  $X$ of~$\Lambda_v v$ in~$\overline{S}$ and a function~$f$ defined on $X^\perp\coloneqq \{ x^\perp \mid x\in X \}$ such that 
\[
S \cap X = \left\{x \mid g(x, e_n) > f(x^\perp) , \ x \in X \right\},\qquad
\partial S \cap X  =\left\{  (x^\perp, f(x^\perp)) \mid  x\in X \right\}.
\]
We can regard  $f$ as a function on  a subset of $\mathbb{R}^{n-1}$ and we find 
\[
\nu_o(x) = \frac{(\nabla_{\mathbb{R}^{n-1}} f(x^\perp),-1)}{|(\nabla_{\mathbb{R}^{n-1}} f(x^\perp),-1)|}
\quad
\text{ for }x\in X.
\]
Let $x_0\in X$ satisfy $\Lambda_v v=x_0=(x_0^\perp, f(x_0^\perp))$.
Recall that $x_0^\perp \neq 0$.
It follows from  $\nu_o(\Lambda_v v)=-e_{n}$ that 
$\nabla_{ \mathbb{R}^{n-1} }f(x_0^\perp)=0$ and
\begin{equation}\label{eq:2.4}
\Lambda_v v-\delta \nu_o(\Lambda_v v)=(x_0^\perp, f(x_0^\perp)+\delta)
\quad \text{for }\delta\in \mathbb{R}.
\end{equation}
Setting  $s\coloneqq (\lambda _v+\Lambda_v)/2$,  
we have $s v\notin \overline{S}$. 
Since $\overline{S}$ is  closed, by  decreasing $\varepsilon>0$ if necessary,
we have
\[
s\Lambda_v^{-1}(\Lambda_v v -\delta \nu_o(\Lambda_v v))
=s\Lambda_v v -\delta s\Lambda_v^{-1}\nu_o(\Lambda_v v))
\notin \overline{S}
\quad \text{for }\delta \in (0,\varepsilon).
\]
For each $\delta \in (0,\varepsilon)$,
this with \eqref{eq:2.3} ensures 
\begin{align*}
s(\delta)&\coloneqq \inf\{ t\in (s \Lambda_v^{-1},1] \mid  t  (\Lambda_v v -\delta  \nu_o(\Lambda_v v))\in S\}\in (s \Lambda_v^{-1},1),\\
w(\delta)&\coloneqq s(\delta) (\Lambda_v v -\delta  \nu_o(\Lambda_v v)\in \partial S.
\end{align*}
We naturally set $s(0)\coloneqq 1$ and, by  decreasing $\varepsilon>0$ if necessary, we have
\[
W\coloneqq \{ w(\delta)\mid \delta\in [0, \varepsilon)\} \subset  \partial S \cap X.
\]
We observe from~\eqref{eq:2.4} and the property 
$\partial S \cap X=\left\{  (x^\perp,f(x^\perp))\ | \ x\in X \right\}$
that
\begin{equation}\label{eq:2.5}
w(\delta)
=s(\delta)(x_0^\perp, f(x_0^\perp)+\delta)
=
\left( s(\delta) x_0^\perp, f( s(\delta) x_0^\perp)\right)
\quad\text{ for }\delta\in [0, \varepsilon).
\end{equation}
Let $\delta\in [0, \varepsilon)$.
Given $t>1$, if it is sufficiently close to $1$,  then 
\[
t w(\delta)=t\left( s(\delta) x_0^\perp, f( s(\delta) x_0^\perp)\right) \in S,
\]
hence $t f( s(\delta) x_0^\perp) >f( ts(\delta) x_0^\perp)$ by definition. 
Then it turns out that 
\begin{align*}
 g( \nu_o( w(\delta) )), w(\delta) )
&=\frac{\left\langle \left(\nabla_{ \mathbb{R}^{n-1}} f(s(\delta) x_0^\perp),-1\right), \left( s(\delta) x_0^\perp, f(s(\delta)x_0^\perp) \right) \right\rangle}
{|(\nabla_{\mathbb{R}^{n-1}} f(x^\perp),-1)|}\\
&=\frac{\langle \nabla_{\mathbb{R}^{n-1}} f(s(\delta)x_0^\perp),  s(\delta)x_0^\perp\rangle-f(s(\delta)x_0^\perp)}
{|(\nabla_{\mathbb{R}^{n-1}} f(x^\perp),-1)|}\\
&=\frac{1}{|(\nabla_{\mathbb{R}^{n-1}} f(x^\perp),-1)|}
\left(\lim_{t\downarrow 1} \frac{f(  ts(\delta) x_0^\perp)-f(  s(\delta) x_0^\perp) }{t-1}-f(s(\delta)x_0^\perp)\right)\\
&=\frac{1}{|(\nabla_{\mathbb{R}^{n-1}} f(x^\perp),-1)|}
\cdot \lim_{t\downarrow 1} \frac{f(  ts(\delta) x_0^\perp)-tf(  s(\delta) x_0^\perp) }{t-1}\leq 0.
\end{align*}
This with \eqref{eq:2.2} for $p=\exp_o(w(\delta))$ leads to
\begin{equation}\label{eq:2.6}
g( \nu_o( w(\delta) )), w(\delta) )=0 \quad\text{ for }\delta\in [0, \varepsilon).
\end{equation}
Since $w(0)=\Lambda_v v=(x_0^\perp, f(x_0^\perp))$ and  $\nu_o(\Lambda_v v)=-e_{n}$,
we find $f(x_0^\perp)=0$.

Notice that, if we assumed the strict inequality $g(\nu(p), \dot{\gamma}_p(1))> 0$ for $p\in \partial \Omega$ 
(as in \cite{DF}*{Theorem~2.1~(ii)}), then the proof  would be finished.

Now we will show that $W$ is contained in a straight line.
By~\eqref{eq:2.5}, there exists an interval $I$ with nonempty interior such that 
\[
W=\{( tx_0^\perp, f(tx_0^\perp))  \mid t\in I \}.
\]
Setting $c(t)\coloneqq tx_0^\perp$ for $t\in I$, we use \eqref{eq:2.6} to have
\[
\left\langle \frac{\left(\nabla_{{\mathbb{R}^{n-1}}} f(c),-1\right)}{|(\nabla_{{\mathbb{R}^{n-1}}} f(c), -1)|} ,  ( c, f(c))\right\rangle
=g\left(\nu_o(  ( c, f(c) ) )  , ( c, f(c) )\right)\\ 
=0
\quad\text{ on } I,
\]
that is, 
\[
\langle \nabla_{{\mathbb{R}^{n-1}}} f(c), c\rangle=f(c)
\quad\text{ on } I.
\]
Differentiating this equality gives
\[
\langle \Hess f|_{c} c', c\rangle+
\langle \nabla_{{\mathbb{R}^{n-1}}} f(c), c' \rangle
 =\frac{d}{dt}\langle \nabla_{{\mathbb{R}^{n-1}}} f(c), c \rangle
 =\frac{d}{dt}f(c)= \langle\nabla_{{\mathbb{R}^{n-1}}} f(c), c' \rangle,
\]
consequently
$ \langle\Hess  f|_{c} c' , c \rangle =0$ on $I$,
which is equivalent to 
$\langle \Hess  f|_{c} c' ,  x_0^\perp \rangle=0 $ on~$I$.
Since we have
\begin{align*}
\frac{d}{dt}\left\langle \nabla_{{\mathbb{R}^{n-1}}} f(c) , x_0^\perp \right\rangle
&= \langle\Hess f|_{c} c' , x_0^\perp\rangle =0
\quad \text{ on }  I,\\
\nabla_{{\mathbb{R}^{n-1}}}f(c(0))&=
\nabla_{{\mathbb{R}^{n-1}}}f(x_0^\perp)=0,
\end{align*}
we find that 
$\langle \nabla_{{\mathbb{R}^{n-1}}} f(c), x_0^\perp\rangle=0$ on $I$ hence 
\[
\frac{d}{dt} f(c)
= \langle \nabla_{{\mathbb{R}^{n-1}}} f(c) ,  c'\rangle
=  \langle\nabla_{{\mathbb{R}^{n-1}}} f(c) ,  x_0^\perp\rangle
=0
\quad \text{ on } I.
\]
This implies that $f(c)\equiv f(x_0^\perp)=0$ on $I$
and 
\[
W
 =\{ (tx_0^\perp, f(x_0^\perp)) \mid t\in I \}
 =\{ t(x_0^\perp, 0) \mid t\in I \}.
 \]
This contradicts the property~\eqref{eq:2.5}, that is,
\[
w(\delta)
=s(\delta)(x_0^\perp, f(x_0^\perp)+\delta)
=s(\delta)  (x_0^\perp, \delta)\in W
\quad
\text{ for }\delta \in [0,\varepsilon).
\]
Thus we conclude the proof of the proposition.
\end{proof}
\begin{remark}
Although the Euclidean counterpart of  Proposition~\ref{theorem:2.6} can be found in~\cite{DF}*{Theorem~2.1},
Proposition~\ref{theorem:2.6} gives an improvement since  originally the inequality is assumed to be strict when the if part  of the assertion is proved.
\end{remark}

\section{{Warped products}}\label{section:3}
We briefly review the definition and some properties of a warped product of an open interval and the unit sphere. 
For further details, we refer to \cite{P} (see in particular Sections {4.2.3 and 4.3.4}). 
 
For a smooth function $\sigma$ on $(0,R)$ with $R\in (0,\infty]$,
the \emph{warped product} $(0,R) \times_\sigma \mathbb{S}^{n-1}(1)$ is the product manifold $(0,R) \times \mathbb{S}^{n-1}(1)$ equipped with the Riemannian metric given by 
\[
g\coloneqq  \mathrm{proj}_{(0,R)}{}^{\ast} dr^2 + (\sigma \circ \mathrm{proj}_{(0,R)})^2  \mathrm{proj}_{\mathbb{S}^{n-1}(1)}{}^{\ast} g_{\mathbb{S}^{n-1}(1)},
\]
where
$\mathrm{proj}_{(0,R)}: (0,R) \times \mathbb{S}^{n-1}(1) \to (0,R),  \mathrm{proj}_{\mathbb{S}^{n-1}(1)}: (0,R) \times \mathbb{S}^{n-1}(1) \to \mathbb{S}^{n-1}(1)$ 
are the natural projections respectively,
and 
$dr^2, g_{\mathbb{S}^{n-1}(1)}$ are  the canonical metric  of $(0,R)$ and $\mathbb{S}^{n-1}(1)$.
The case that  $B_o(R)\setminus\{o\}$ is isometric to  $(0,R) \times_\sigma \mathbb{S}^{n-1}(1)$ is formulated as follows.
\begin{definition}
Let $o\in M$ and $R\in (0,\infty]$.
We say that a map $\xi\colon B_o(R)\to \mathbb{R}^{n}$  is a \emph{rotationally symmetric coordinate system at $o$}
if  $(\xi, B_o(R))$ forms a coordinate system at~$o$ with $\xi(o)=0$ and there exists $\psi\colon [0,   {\rho}) \to (0,\infty)$ such that 
\[
g_{ij}(p)=\psi(|\xi(p)|)^2\delta_{ij}\quad \text{ for }p\in   {B_o(R)},
\qquad
\text{where }\rho\coloneqq \sup_{p\in B_o(R)}|\xi(p)|.
\]
We refer to $\psi$ as the \emph{radial conformal factor} of $(B_o(R),\xi)$.
\end{definition}
Let  $\xi\colon B_o(R)\to \mathbb{R}^{n}$ be a rotationally symmetric coordinate system at $o$ with radial conformal factor $\psi$.
For a domain $\Omega\subset B_o(R)$, 
we see that $\Omega$ is  a starshaped neighborhood of $o$  if and only if~$\xi(\Omega)$ is a starshaped set about $0$ in $\mathbb{R}^n$.
Notice that the function $x \mapsto \psi(|x|)^2$  is smooth on $\xi(B_o(R))$
and
\begin{equation}\label{eq:3.1}
\lim_{|x| \to 0}\nabla_{\mathbb{R}^{n}} \psi(|x|)
=\lim_{|x| \to 0} \psi'(|x|){\frac{x}{|x|}}=0
\end{equation}
holds (see for instance \cite{P}*{Subsection~4.3.4}).

For  $u\in C^{\infty}(B_o(R))$, 
setting $U\coloneqq u \circ \xi^{-1}$,
we have
\begin{align*}
\Delta_{q,M}u
=
\frac{1}{\psi(|\xi|)^{q}}\left[ \left( \Delta_{q,\mathbb{R}^n} U\circ \xi\right) +
(n-q)
\left( |\nabla_{\mathbb{R}^{n}}U | \circ \xi \right)^{q-2}
 \left\langle \nabla_{\mathbb{R}^n} \log \psi(|\cdot|),\left(\nabla_{\mathbb{R}^n} U\right)\circ \xi \right\rangle 
\right],
\end{align*}
in $B_o(R)\setminus\{o\}$, and 
\[
\Delta_{q,M} u(o)
=\frac{1}{\psi(0)^q}\Delta_{q,\mathbb{R}^{n}} U (0).
\]
In particular, if $q=2$ we have
\begin{equation}\label{eq:3.2}
\Delta_M u 
=\frac{1}{\psi(|\xi|)^2} 
\left[  \left( \Delta_{\mathbb{R}^{n}} U\right) \circ \xi 
+(n-2)\left\langle  \nabla_{\mathbb{R}^n} \log \psi(|\cdot|), ( \nabla_{\mathbb{R}^{n}} U) \circ \xi \right\rangle
\right]
\end{equation}
in  $B_o(R)\setminus\{o\}$, and
\[
\Delta_{M} u(o)=\frac{1}{\psi(0)^2}\Delta_{\mathbb{R}^{n}} U (0).
\]
\begin{example}
For $\kappa\in \mathbb{R}$, 
let $\mathbb{M}^n_{\kappa}$ be an $n$-dimensional simply connected space form  of constant curvature~$\kappa$.
\begin{itemize}
\setlength{\leftskip}{-15pt}
\item
Needless to say, $\mathbb{M}^n_0=\mathbb{R}^n$ and 
$\xi=\lambda^{-1} \cdot \mathrm{id}_{\mathbb{R}^n}\colon\mathbb{R}^n \to \mathbb{R}^n$ provides 
a rotationally symmetric coordinate system at $0$ for any $\lambda \neq 0$. 
Then $\psi\equiv \lambda$. 
\item
For $\kappa>0$, set $r\coloneqq \kappa^{-1/2}$.
Then $\mathbb{M}^n_\kappa$ is the standard $n$-dimensional sphere of radius~$r$, that is,
\[
\mathbb{M}^n_\kappa=\{ p\in \mathbb{R}^{n+1}\mid  |p|=r\}.
\]
Then $B_{r e_{n+1}}(r\pi)=\mathbb{M}^n_\kappa\setminus\{-r e_{n+1}\}$ holds.
We denote by $\xi\colon B_{r e_{n+1}}(r\pi)\to \mathbb{R}^n$ the stereographic projection from the south pole,  that is,
\[
\xi(p)\coloneqq \frac{r}{r+\langle p, e_{n+1}\rangle}p^{\perp}, 
\]
where $p^\perp\in \mathbb{R}^n$ satisfies $p=(p^\perp, \langle p, e_{n+1}\rangle)$.
Then this is a rotationally symmetric coordinate system at $r e_{n+1}$ 
and its radial conformal factor is given by
\[
\psi(t)\coloneqq \frac{2r^2}{r^2+t^2}:[0,\infty)\to\mathbb{R}.
\]
\item
In the case of $\kappa<0$, namely the hyperbolic space, we use the Poincar\'e disk model.
Let $r\coloneqq |\kappa|^{-1/2}$ and
$\mathbb{B}^n(r)$  be the disk in $\mathbb{R}^n$ centered at the origin with radius $r$.
Consider the inclusion map $\xi\colon \mathbb{B}^n(r) \to \mathbb{R}^n$.
For the Riemannian metric on $\mathbb{B}^n(r)$, choose 
\[
g_p\left(\frac{\partial}{\partial \xi^i}, \frac{\partial}{\partial \xi^j} \right)\coloneqq \frac{4r^4}{(r^2-|p|^2)^2} \delta_{ij}.
\]
Then $B_{0}(\infty)=\mathbb{B}^n(r)$
and  $\xi\colon \mathbb{B}^n(r) \to \mathbb{R}^n$ is a rotationally symmetric coordinate system at $o$ 
and its radial conformal factor is given by
\[
\psi(t)\coloneqq \frac{2r^2}{r^2-t^2}:[0,r)\to \mathbb{R}.
\]
\end{itemize}
\end{example}

\section{Viscosity solutions}\label{section:4}
Viscosity solutions were introduced by Crandall and Lions for  Euclidean space in the 1980s and immediately became very popular.
The literature about viscosity solutions is huge, and for simplicity we only refer to the classical  reference~\cite{CIL}.
Here we briefly recall some basic definitions and results as needed in this paper.

\begin{definition}
Let $D$ be a domain in $\mathbb{R}^n$.
Let $h:D\to \mathbb{R}$ be upper (resp.\,lower) semicontinuous. 
\begin{itemize}
\setlength{\leftskip}{-15pt}
\item
For $x \in D$,
we say that a $C^2$-function $\varphi$ defined on an open neighborhood of $x$  \emph{touches~$h$ at $x$ from above $($resp.\,below$)$}
if $\varphi(x)=h(x)$ and there exists an open neighborhood $O$ of  $x$  in $D$ such that 
\[
 \varphi\geq h\quad(\text{resp.\,}\varphi\leq h)\quad\text{ in }O.
\]
\item
Let $F:D \times \mathbb{R}\times \mathbb{R}^{n}\to\mathbb{R}$.
We say that $h$ is a \emph{viscosity subsolution} (resp.~\emph{supersolution}) to 
\begin{equation}\label{eq:4.1}
\Delta_{q,\mathbb{R}^{n}}u =F(\cdot, u,\nabla_{\mathbb{R}^n} u)\quad\text{ in } D
\end{equation}
if, for $x\in D$, 
any $C^2$-function $\varphi$  touching $h$ at $x$ from above (resp.\,below) satisfies
\[
\qquad \Delta_{q,\mathbb{R}^{n}} \varphi(x)\geq F(x,\varphi(x),\nabla_{\mathbb{R}^{n}} \varphi(x))
\quad
\left(\text{resp.}\,\Delta_{q,\mathbb{R}^{n}} \varphi(x)\leq F(x,\varphi(x),\nabla_{\mathbb{R}^{n}} \varphi(x))\right).
\]
\item
A \emph{viscosity solution} to \eqref{eq:4.1} is  a continuous function which is subsolution and supersolution at the same time.
\end{itemize}
\end{definition}
It is easily seen that, 
for $C^2$-functions $\varphi$ and $h$ defined on an open neighborhood of $x\in D$, if $\varphi$ touches~$h$ at $x$ from above (resp.\,below),
then 
\[
\nabla_{\mathbb{R}^n} \varphi (x)= \nabla_{\mathbb{R}^n}h (x)
\]
and $\mathop{\mathrm{Hess}_{\mathbb{R}^n}} \varphi (x)- \mathop{\mathrm{Hess}_{\mathbb{R}^n}}h (x)$ is nonnegative (resp.\,nonpositive) definite,
whence
\[
\Delta_{q,\mathbb{R}^n} \varphi (x) \geq  \Delta_{q,\mathbb{R}^n}h (x)
\quad 
\left(\text{resp.}\, \Delta_{q,\mathbb{R}^n} \varphi (x) \leq  \Delta_{q,\mathbb{R}^n}h (x)\right).
\]
Notice that if a viscosity subsolution (resp.\,supersolution) is of class $C^2$, then it is a classical subsolution (resp.\,supersolution).

The theory of viscosity solutions in Riemannian manifolds essentially resembles that in Euclidean space and 
the relevant definitions are just the exact translation of the ones given above for the Euclidean setting (see for instance~\cites{AFS,GP}), then we do not repeat here the details.
Indeed, we use only the theory of viscosity solutions in Euclidean space.

We close this section by recalling the weak comparison principle.
\begin{definition} 
We say that the \emph{weak comparison principle} holds for equation \eqref{eq:4.1} 
if the following holds true.
\begin{itemize}
\setlength{\leftskip}{-15pt}
\item
If $u$ is a supersolution and $h$ is a viscosity subsolution to \eqref{eq:4.1} in $D$ such that $h\leq u$ on $\partial D$, then $h\leq u$ in $D$.
\end{itemize}
\end{definition}
%

\section{Main results}\label{section:5}
Throughout this section, 
let  $R\in (0,\infty]$, $\xi\colon B_o(R)\to \mathbb{R}^{n}$  be  a rotationally symmetric coordinate system  at $o$ with radial conformal factor $\psi$.
Let $\Omega_0, \Omega_1$ be two  starshaped neighborhoods of $o$ 
such that $\overline{\Omega_1}\subset \Omega_0 \subset B_o(R)$, $\Omega_0$ is regular, and $\xi(\Omega_0)$ is bounded.
Set
\[
 \Omega\coloneqq \Omega_0\setminus\overline{\Omega_1}.
\]

Let us consider the following elliptic boundary value problem
\begin{equation}\label{eq:5.1}
\begin{cases}
\Delta_{M} u=0 &\text{in }\Omega,\\
u=0 & \text{on }\partial \Omega_0,\\
u=1& \text{in }\overline{\Omega_1}.
\end{cases}
\end{equation}
Applying the standard theory of elliptic equation (see for instance \cite{GT}*{Chapter 6}) we see that a unique classical solution $u$ exists
and that $0\leq u \leq 1$ on $\overline{\Omega_0}$.
Our main result consists in proving that every superlevel set 
\[
L^+_{\ell}\coloneqq \{p\in\Omega_0 \ | \ u(p)>\ell\} \quad\text{for }\ell\in(0,1)
\]
is a starshaped neighborhood of $o$,  that is, $\log_o (L^+_{\ell})$ is a starshaped  set about $0$ in $T_oM$.
In analogy with the notion of quasi-concave function (a function whose superlevel sets are all convex), 
we introduce the following definition.
\begin{definition}
Let $o\in M$.
A  function on a domain of $M$  is said \emph{quasi-starshaped} about~$o$
if its superlevel sets are all starshaped neighborhoods of $o$. 
\end{definition}
According to this definition, we will prove that a classical solution $u$ to \eqref{eq:5.1} is quasi-starshaped. 
To this aim, we define a function on $\Omega_0$ by 
\begin{equation}\label{eq:5.2}
u^\ast(p)\coloneqq \sup\{u(\gamma_p(t)) \mid  t\in[1,T_{\Omega_0,p})\} \quad 
\text{ for }p\in \Omega_0.
\end{equation}
Notice that $u^\ast$ is quasi-starshaped about $o$ and $u^\ast\geq u$.
Indeed, $u^\ast$ is the smallest quasi-starshaped function greater or equal to $u$ then  we call  $u^\ast$ the quasi-starshaped envelope of $u$. 
Straightforwardly, we have
\[
u\leq u^\ast\leq 1\quad \text{in }\overline\Omega_0,\qquad u^\ast=1\quad\text{in }\overline\Omega_1,\qquad u^\ast=0\quad \text{on }\partial\Omega_0\,\text{ if }\partial\Omega_0\neq\emptyset.
\]
Now we are ready to state our main theorem.
\begin{theorem}\label{theorem:5.2}
Let $\Omega_0$, $\Omega_1$ and $\Omega$ be as stated in the preamble of this section.
Then any classical solution $u$ to problem~\eqref{eq:5.1} is quasi-starshaped about~$o$,  i.e., every superlevel set of $u$ is a starshaped neighborhood of $o$. 
\end{theorem}
\begin{proof}
The idea is to prove that $u$ must coincide with the function $u^\ast$ given by~\eqref{eq:5.2}, using the weak comparison principle. 
To simplify the argument, we first translate problem \eqref{eq:5.1} into a problem in  Euclidean space.

Let 
\[
X_0\coloneqq \xi(\Omega_0), \qquad
X_1\coloneqq \xi(\Omega_1), \qquad
X\coloneqq \xi(\Omega)=X_0\setminus\overline{X_1}.
\]
We also set 
\[
\Psi(x)\coloneqq \psi(|x|) \quad \text{for }x\in X_0.
\]
For $x \in X_0$, 
define $T_x \in [1,\infty]$ by 
\[
T_x\coloneqq \sup\{  t\geq1 \mid t x \in \overline{X_0} \},
\]
similarly to  $T_{\Omega_0,p}$.
Then $T_x >1$ for $x \in X$ and  $T_{t x}= t^{-1} T_x$ for $x\in X_0$ with $x \neq 0$ and $t\in (0, T_x)$.
Set
\[
U\coloneqq u \circ \xi^{-1}: \overline{X_0}\to \mathbb{R}.
\]
Then  $U\in C(\overline{X_0})\cap C^2(X_0)$ and  $U$ solves
\begin{equation*}
\begin{cases}
\Delta_{\mathbb{R}^{n}} U+
(n-2)\left\langle \nabla_{\mathbb{R}^n} \log \Psi, \nabla_{\mathbb{R}^{n}}  U\right\rangle =0 &\text{ in }X,\\
U=0 &\text{on }\partial X_0,\\
U=1&\text{on }\overline{X_1},
\end{cases}
\end{equation*}
by \eqref{eq:5.1} together with \eqref{eq:3.2}.
We prove that  $u$ is quasi-starshaped about $o$ if and only if $U$ is quasi-starshaped about $0$.
First, we set $U^\ast\coloneqq u^\ast\circ \xi^{-1}: \overline{X_0}\to \mathbb{R}$, that is,
\[
U^\ast(x):
=\sup\left\{ U(t x) \bigm|  t \in [1,T_x) \right\}
\quad
\text{ for }x\in \overline{X_0}.
\]
For $x \in X_0$,
since $U(x)>0$ and $U(tx)\to 0$ as $t\uparrow T_x$
we find $t_x \in [1,T_x)$ such that $U^\ast(x)=U(t_x x)$. 
Then $U^\ast$ is continuous on $\overline{X_0}$ and satisfies 
\begin{equation*}
0\leq U^\ast\leq 1\quad\text{in }\overline{X_0},
\qquad 
U^\ast=1\quad\text{on }\overline{X_1},
\qquad 
U^\ast=0
\quad \text{on }\partial X_0.
\end{equation*}
Fix $\ell\in (0,1)$.
It is easily seen that every superlevel set 
\[
L^+_{\ell}\coloneqq\{x \in X_0 \mid U^\ast (x)>\ell\}
\]
is a starshaped set  about $0$ in $\mathbb{R}^{n}$.
Indeed, it follows from  $o\in \Omega_1$ that $U^\ast(0)=1$ and hence  $0\in L^+_{\ell}$.
For $x \in L^+_{\ell}$ and $t\in[0,1]$, 
we find $T_x =tT_{tx}$, which yields 
\begin{align*}
U^{\ast}( tx)
&=\max\left\{ U( \tau  t x  )  \bigm|  \tau \in [1,T_{tx}] \right\}
=\max\left\{ U(\tau x ) \bigm|  \tau \in [t, T_x] \right\}\\
&\geq
\max\left\{ U(\tau x) \bigm|  \tau \in [1,T_x] \right\}
=U^\ast (x)\\
&>\ell.
\end{align*}
Thus $ tx \in L^+_{\ell}$.
In addition, we notice that 
if  $\overline{U}   \in C(\overline{X_0})$ is {quasi-starshaped} about $0$  and $U\leq \overline{U}$ on $\overline{X_0}$, 
then $U^\ast \leq \overline{U}$ holds on {$X_0$}. 
%

We show that $U^\ast$ is  a viscosity subsolution to
\begin{equation}\label{eq:5.3}
\Delta_{\mathbb{R}^{n}} h
=-(n-2)\left\langle \nabla_{\mathbb{R}^n} \log \Psi, \nabla_{\mathbb{R}^{n}}  h\right\rangle
\quad \text{in } X.
\end{equation}
Let $x \in X$ be such that $U^\ast(x)=U(x)$.   
Then any $C^2$-function $\varphi$ touching  $U^\ast$ at $x$ from above  also touches $U$ at $x$ from above  
and it holds
\begin{align*}
\Delta_{\mathbb{R}^{n}} \varphi(x)
\geq 
\Delta_{\mathbb{R}^{n}} U(x)
=-(n-2) \left\langle \nabla_{\mathbb{R}^n} \Psi(x), \nabla_{\mathbb{R}^{n}}  U(x)\right\rangle
=-(n-2) \left\langle \nabla_{\mathbb{R}^n} \Psi(x), \nabla_{\mathbb{R}^{n}}  \varphi(x)\right\rangle.
\end{align*}
Let $x \in X$ be such that $U^\ast(x)> U(x)$. 
There exists $t_x \in (1,T_x)$  such that $U^\ast(x)=U(t_x x)$.
Setting $B_x\coloneqq \{ y\in X \mid t_x y \in X \}$, which is a neighborhood of $x$,
we define a function 
$
\widetilde{U}:B_x\to\mathbb{R}$ by
\[
\widetilde{U}(y)\coloneqq U( t_x y)
\quad
\text{ for } y\in B_x.
\]
By construction, we have
\[
\widetilde{U}\leq U^\ast \quad\text{ in } B_x\quad \text{ and }\quad \widetilde{U}(x)=U^\ast(x).
\]
Hence, if a $C^2$-test function $\varphi$  touches  $U^\ast$ at ${x}$ from above, then it  touches  also $\widetilde{U}$ at ${x}$ from above, whence 
\begin{align}
\label{eq:5.4}
\nabla_{\mathbb{R}^{n}} \varphi (x)
&=\nabla_{\mathbb{R}^{n}}  \widetilde{U}(x)
=t_x \nabla_{\mathbb{R}^{n}} U (t_x x), \\ \notag
 \Delta_{\mathbb{R}^{n}} \varphi (x)
 &\geq  \Delta_{\mathbb{R}^{n}} \widetilde{U}(x)
=t_x^2  \Delta_{\mathbb{R}^{n}} U (t_{x} x).
\end{align}
On the other hand, 
we observe from the choice of $t_{x}$ that 
\begin{align}\label{eq:5.5}
\left \langle x, \nabla_{\mathbb{R}^{n}}U(t_{x} x) \right\rangle 
=\frac{\partial}{\partial  t}\bigg|_{t_{x}} U(t x) =0.
\end{align}
These provides
\begin{align*}
\Delta_{\mathbb{R}^{n}} \varphi (x)\geq
t_x^2  \Delta_{\mathbb{R}^{n}} U (t_{x} x)
&=-t_x^2 (n-2)
\left\langle \nabla_{\mathbb{R}^n} \log \Psi(t_x x), \nabla_{\mathbb{R}^{n}}  U(t_x x)\right\rangle\\
&=0\\
&=-(n-2)\left\langle \nabla_{\mathbb{R}^n} \log \Psi(x), \nabla_{\mathbb{R}^{n}}  \varphi(x)\right\rangle,
\end{align*}
where we use the property
\[
\nabla_{\mathbb{R}^n} \log \Psi(y)=\frac{\psi'(|y|)}{\psi(|y|)} \cdot \frac{y}{|y|}
\quad \text{for }y\in X.
\]
Thus $U^\ast$ is a viscosity subsolution to~\eqref{eq:5.3} in $X$  as desired.

Since the weak comparison principle holds for~\eqref{eq:5.3} (see for instance \cite{CIL}*{Theorem~3.3}), 
we obtain $U^\ast\leq U$ in $X$  hence $U=U^\ast$ in $X$.
Then we see that $U$ is quasi-starshaped about $0$, that is, $u$ is quasi-starshaped. The proof is complete.
\end{proof}

Similarly to Problem \eqref{eq:5.1}, we can treat the following problem
\begin{equation}\label{eq:5.6}
\begin{cases}
\Delta_{q,M} u=0 &\text{in }\Omega,\vspace{5pt}\\
u=0 & \text{on }\partial \Omega_0,\\
u=1& \text{on }\overline{\Omega_1},
\end{cases}
\end{equation}
and prove the following result.
%
%
\begin{theorem}\label{theorem:5.3}
Let $\Omega_0$, $\Omega_1$ and $\Omega$ be as stated in the preamble of this section.
Then any differentiable viscosity solution $u$ to problem~\eqref{eq:5.6} is {quasi-starshaped} about~$o$,  i.e., every superlevel set of $u$ is a starshaped neighborhood of $o$. 
\end{theorem}
\begin{proof}
The idea is the same, that is to prove that $u$ must coincide with its quasi starshaped envelope $u^\ast$,  using the weak comparison principle. 
We first translate problem \eqref{eq:5.1} into a problem in Euclidean space.
We use the same notations $X_0$, $X_1$, $X$, $\Psi$,  $t_{x}$ and so on.
Also, setting $U:=u\circ \xi^{-1}$, we find that 
$U\in C(\overline{X_0})$ and  $U$ solves
\begin{equation*}
\begin{dcases}
 \Delta_{q,\mathbb{R}^n} U +
(n-q)
 |\nabla_{\mathbb{R}^{n}}U | ^{2-q}
 \left\langle  \nabla_{\mathbb{R}^n} \log \Psi,\nabla_{\mathbb{R}^n} U \right\rangle 
=0 &\text{in }X,\\
U=0 &\text{on }\partial X_0,\\
U=1&\text{on }\overline{X_1},
\end{dcases}
\end{equation*}
in the viscosity sense.
Since the proof works exactly in the same way to the proof of Theorem~\ref{theorem:5.2}, 
we prove that $U^\ast\coloneqq u^\ast\circ \xi^{-1}$ is a viscosity subsolution to
\begin{equation}\label{eq:5.7}
\Delta_{q,\mathbb{R}^{n}} h
=-(n-q)|\nabla_{\mathbb{R}^n} h|^{q-2}\left\langle \nabla_{\mathbb{R}^n} \log \Psi, \nabla_{\mathbb{R}^{n}}  h\right\rangle
\quad \text{in } X.
\end{equation}
Again, \eqref{eq:5.5} holds, which is crucial. 
For the function $\widetilde{U}(y):=U(t_x y)$ defined around $x$, 
\begin{equation}\label{eq:5.8}
\Delta_{q,\mathbb{R}^{n}} \widetilde{U}(x)
=t_x^q  \Delta_{q,\mathbb{R}^{n}} U(t_{x} x)
\end{equation}
holds in the viscosity sense,
hence $U^\ast$ becomes a viscosity subsolution to~\eqref{eq:5.7}.
Since  the weak comparison principle holds for~\eqref{eq:5.7} (see for instance \cite{CIL}*{Theorem~3.3}),
we obtain $U=U^\ast$ in $X$.
Then  $u$ is quasi-starshaped as desired.
\end{proof}

Next, as a further generalization of problem~\eqref{eq:5.1}, 
we consider the following elliptic boundary value problem
\begin{equation}
\label{eq:5.9}
\begin{cases}
\Delta_{q,M} u=f(p,u, \nabla_M u) & \text{ in }\Omega,\\
u=0 & \text{on }\partial \Omega_0,\\
u=1& \text{in\ }\overline{\Omega_1},\\
\end{cases}
\end{equation}
where $f$ is a nonnegative H\"older continuous function on 
\[
\mathcal{X}_\Omega\coloneqq \{  (p,s,v)\mid p\in \overline\Omega, s\in [0,1],  v\in T_pM\}.
\]
\begin{theorem}\label{theorem:5.4}
Let $\Omega_0$, $\Omega_1$ and $\Omega$ be as stated in the preamble of this section.
Assume that $f$ is non-decreasing with respect to the second variable and
\begin{equation}\label{eq:5.10}
\left(\frac{|\xi(\gamma_p(t))|\cdot \psi(|\xi(\gamma_p(t))|)}{|\xi(p)|\cdot \psi(|\xi(p)|)}\right)^q
f\left(\gamma_p(t), s,  \frac{|\xi(p)|}{|\xi(\gamma_p(t))|} V(t) \right)\geq f(p,s, v)
\end{equation}
for $(p,s,v)\in  \mathcal{X}_\Omega$ and $t\geq 1$ with $\gamma_p(t)\in \Omega$,
where $V$ is the parallel vector field along the curve~$\gamma_p:[0, T_{\Omega_0,p})\to M$ with $V(1)=v$.
Then any differentiable viscosity solution  $u$ to problem~\eqref{eq:5.9} is quasi-starshaped about~$o$.
\end{theorem}
\begin{proof}
The strategy of the proof is the same as that of Theorem~\ref{theorem:5.2}.
We use the same notations $X$, $\Psi$, $U$, $U^{\ast}$, $t_{x}$, $\widetilde{U}$ and so on.
For a vector $(v^1, \cdots, v^n)\in \mathbb{R}^n$, define the tangent vector $v \in T_p M$ for $p\in \Omega$ by
\[
v\coloneqq \sum_{i=1}^n\frac{v^i}{ \Psi(\xi(p))}  \frac{\partial}{\partial \xi^i}\bigg|_p.
\]
Then the parallel vector filed along
the curve $\gamma_p:[0, T_{\Omega_0,p})\to M$ with $V(1)=v$ is given by
\[
V(t)= \sum_{i=1}^n  \frac{v^i}{\Psi(\xi(\gamma_p(t)))} \frac{\partial}{\partial \xi^i}\bigg|_{\gamma_p(t)}
\quad
\text{for }t\in [0, T_{\Omega_0,p}).
\]
In addition, we define $F\colon X \times [0,1] \times \mathbb{R}^n \to \mathbb{R}$ by 
\[
F\left(x, s, (v^1, \cdots, v^n) \right)\coloneqq f\left(\xi^{-1}(x),s,  v\right).
\]
Then the relation~\eqref{eq:5.10} is rephrased as
\begin{equation}\label{eq:5.11}
\tau^q \Psi(\tau x)^q F\left(\tau x, s,  \frac{1}{\tau} v \right) \geq \Psi(x)^q F(x,s, v)
\end{equation}
for $(x,s)\in X\times[0,1]$ and $\tau\geq 1$ with $\tau x \in X$.
We also find that  $U$ solves
\begin{equation}\label{eq:5.12}
\Delta_{q,\mathbb{R}^{n}} U+(n-q) |\nabla_{\mathbb{R}^n} U|^{q-2}\left\langle\nabla_{\mathbb{R}^n} \Psi, \nabla_{\mathbb{R}^{n}}  U\right\rangle 
=\Psi^qF(\cdot , U,\nabla_{\mathbb{R}^n} U)  
\quad
\text{in }X
\end{equation}
in the viscosity sense.
Since the weak comparison principle holds (see \cite{CIL}*{Theorem 3.3}), 
by a similar argument to the proof of Theorem~\ref{theorem:5.2}, 
it is enough to show that  $U^\ast$ is  a viscosity subsolution to~\eqref{eq:5.12}.

Fix $x \in X$. 
If $U^\ast(x)=U(x)$,  
then any $C^2$-function $\varphi$ touching  $U^\ast$ at $x$ from above, also touches $U$ from above at $x$ and 
it holds
\begin{align*}
 \Delta_{q,\mathbb{R}^{n}} \varphi (x)
\geq -(n-q) |\nabla_{\mathbb{R}^n} \varphi|(x)^{q-2}\left\langle\nabla_{\mathbb{R}^n} \log \Psi(x), \nabla_{\mathbb{R}^{n}}  \varphi(x)\right\rangle
+
\Psi(x)^qF\left(x, \varphi(x), \nabla_{\mathbb{R}^n} \varphi(x) \right)
\end{align*}
as in the proof in Theorem~\ref{theorem:5.2}.
Next assume  $U^\ast(x)> U(x)$.
If a $C^2$-test function $\varphi$  touches  $U^\ast$ at ${x}$ from above, then it  touches  also $\widetilde{U}$ at ${x}$ from above, whence 
\begin{align*}
 &  \Delta_{q,\mathbb{R}^{n}} \varphi (x)\\
 & \geq \Delta_{q,\mathbb{R}^{n}} \widetilde{U} (x)
 =t_x^q  \Delta_{q,\mathbb{R}^{n}} U(t_x x)\\
 & =t_x^q
\big[
-(n-q) |\nabla_{\mathbb{R}^{n}}U | (t_xx)^{q-2}
 \left\langle \nabla_{\mathbb{R}^n} \log \Psi(t_xx),  \nabla_{\mathbb{R}^n} U (t_x x) \right\rangle \\
& \qquad\quad+
\Psi(t_x  x )^q F\left( t_x x , U(t_x x), \nabla_{\mathbb{R}^n} U (t_x x)\right)  
\big]\\
 & =t_x^q
\left[
0+ \Psi (t_x  x )^q F\left( t_x x , \varphi(x), \frac{1}{t_x}\nabla_{\mathbb{R}^n} \varphi (x)\right)  
\right]\\
& \geq \Psi(x)^qF\left(x, \varphi(x),  \nabla_{\mathbb{R}^n} \varphi (x) \right)\\
 & =
-(n-q) |\nabla_{\mathbb{R}^n} \varphi|(x)^{q-2}\left\langle\nabla_{\mathbb{R}^n} \log \Psi(x), \nabla_{\mathbb{R}^{n}}  \varphi(x)\right\rangle
+
\Psi(x)^qF\left(x, \varphi(x), \nabla_{\mathbb{R}^n} \varphi(x) \right)
\end{align*}
by \eqref{eq:5.8}, \eqref{eq:5.5}, \eqref{eq:5.4}, and \eqref{eq:5.11}.
This means that $U^\ast$ is a viscosity subsolution to~\eqref{eq:5.12} as desired. 
Then, similarly to the proof of Theorem~\ref{theorem:5.2}, 
we conclude that $U$ is quasi-starshaped about~$0$, 
that is, $u$ is  quasi-starshaped  about~$o$. 
Thus Theorem~\ref{theorem:5.4} follows.
\end{proof}
\begin{remark}
By the geodesic equation, we find 
\[
\psi(|\xi \left( \gamma_p(t) \right)| )\cdot \frac{d}{dt}|\xi \left( \gamma_p(t) \right)|  =1.
\]
Setting $\alpha_p(t)\coloneqq \log |\xi(\gamma_p(t)|$, 
we can  rewrite  \eqref{eq:5.10} as
\begin{equation*}
\left(\frac{\alpha_p'(1)}{\alpha_p'(t)}\right)^q
f\left(\gamma_p(t), s,  \frac{|\xi(p)|}{\xi(\gamma_p(t))|} V(t) \right)\geq f(p,s, v).
\end{equation*}
Assume the concavity of $\alpha_p$, 
or equivalently, 
\[
\frac{\alpha_p'(1)}{\alpha_p'(t)} \geq 1
\qquad\text{for }t\geq1,
\]
which is satisfied if $M=\mathbb{M}^n_\kappa$ with either $\kappa \leq 0$ or $\kappa>0$ with $R\leq \kappa^{-1/2}\pi/2$.
If
\[
 f(\gamma_p(t),s)\geq f(p,s)\quad \text{for }t\geq1
\]
holds, 
then \eqref{eq:5.10} is satisfied.
\end{remark}
\begin{ack}
The first and third authors were supported in part by JSPS KAKENHI Grant Number 19H05599. 
The second author was supported in part by  INdAM through a GNAMPA Project and by 
the project "Geometric-Analytic Methods for PDEs and Applications (GAMPA)", funded by European Union --Next Generation EU  within the PRIN 2022 program 
(D.D. 104 - 02/02/2022 Ministero dell'Universit\`a e della Ricerca).
The third author was also supported in part by JSPS KAKENHI Grant Number 19K03494.
\end{ack}
 
\end{document}